\documentclass[reqno]{amsart}
\usepackage{amsmath}
\usepackage{amssymb}
\usepackage{amscd}
\usepackage{color}
\usepackage[all]{xy}

\newcommand{\fm}{\mathfrak{m}}
\newcommand{\GD}{\operatorname{GD}}
\newcommand{\gGD}{\operatorname{gGD}}

\newcommand{\Max}{\operatorname{Max}}
\newcommand{\Spec}{\operatorname{Spec}}
\newtheorem{theorem}{Theorem}[section]
\newtheorem{lemma}[theorem]{Lemma}
\newtheorem{corollary}[theorem]{Corollary}
\newtheorem{proposition}[theorem]{Proposition}

\newtheorem{example}[theorem]{Example}
\newtheorem{remark}[theorem]{Remark}

\begin{document}

\title[On graded going-down domains, II]
{On graded going-down domains, II}
\author[P. Sahandi and N. Shirmohammadi]
{Parviz Sahandi and Nematollah Shirmohammadi}

\address{(Sahandi) Department of Pure Mathematics, Faculty of Mathematics, Statistics and Computer Science,  University of Tabriz, Tabriz, Iran.} \email{sahandi@ipm.ir}
\address{(Shirmohammadi) Department of Pure Mathematics, Faculty of Mathematics, Statistics and Computer Science,  University of Tabriz, Tabriz, Iran.} \email{shirmohammadi@tabrizu.ac.ir}



\begin{abstract}
In this paper we consider the graded going-down property of graded integral domains in pullbacks. It then enables us to give original examples of these domains.
\end{abstract}

\maketitle

\section{Introduction}

Let $R\subseteq T$ be an extension of (commutative) integral domains and $P$ be a prime ideal of $R$. The extension $R\subseteq T$ is said to satisfy \emph{going-down to $P$} in case if $P\subseteq P_1$ are prime ideals of $R$ and $Q_1$ is a prime ideal of $T$ satisfying $Q_1\cap R=P_1$, then there exists a prime ideal $Q$ of $T$ such that $Q\subseteq Q_1$ and $Q\cap R=P$ \cite{d97, mc74}. The extension $R\subseteq T$ is said to satisfy \emph{going-down} (for short, $\GD$) if, $R\subseteq T$ satisfies going-down to $P$ for every prime ideal $P$ of $R$ \cite[Page 28]{k74}. The integral domain $R$ is said to be a \emph{going-down domain} (for short, a $\GD$ domain) if, for every overring $T$ of $R$, the extension $R\subseteq T$ satisfies $\GD$ \cite{d74, dp76}. It is known that every Pr\"{u}fer domain and every Noetherian integral domain of (Krull) dimension less than or equal to 1  are $\GD$ domains \cite[Corollary 4 and Proposition 7]{d73}. In \cite{d09} Dobbs examined the $\GD$ domain property in pullback of integral domains. Because of its ability to construct examples and counterexamples, pullback construction is very important in Multiplicative Ideal Theory.

Recently the graded version of $\GD$ domains was introduced and investigated in \cite{ss23}. Let $R\subseteq T$ be an extension of $\Gamma$-graded domains for a (nonzero) torsionless commutative cancellative monoid $\Gamma$ and $P$ be a homogeneous prime ideal of $R$. The extension $R\subseteq T$ is said to satisfy \emph{graded going-down to $P$} in case if $P\subseteq P_1$ are homogeneous prime ideals of $R$ and $Q_1$ is a homogeneous prime ideal of $T$ satisfying $Q_1\cap R=P_1$, then there exists a homogeneous prime ideal $Q$ of $T$ such that $Q\subseteq Q_1$ and $Q\cap R=P$. We say that $R\subseteq T$ satisfies \emph{graded going-down} (for short, $\gGD$) if, $R\subseteq T$ satisfies graded going-down to $P$ for every homogeneous prime ideal $P$ of $R$. The graded domain $R$ is said to be a \emph{graded going-down domain} (for short, a $\gGD$ domain) if, for every homogeneous overring $T$ of $R$, the extension $R\subseteq T$ satisfies $\gGD$. It is shown that one may restrict the test overrings $T$ to simple homogeneous overrings or to gr-valuation overrings of $R$ \cite[Theorem 3.12]{ss23}. It is also proved that $\gGD$ property is stable under the formation of rings of fractions and factor domains (see \cite[Proposition 4.1 and Corollary 4.8]{ss23}).

The present paper is a sequel to \cite{ss23}. The purpose here is to study the behaviour of the concept of $\gGD$ domains in pullbacks of graded domains and to give some new examples of $\gGD$ domains. To this, in Section 2, we recall the definition and properties of graded pullbacks and prove some preliminary results. In Section 3, we prove our main result of the paper:

\begin{theorem}
Let $T=\bigoplus_{\alpha \in \Gamma}T_{\alpha}$ be a $\Gamma$-graded integral domain, $M =\bigoplus_{\alpha \in \Gamma}M_{\alpha}$ be a maximal homogeneous ideal of $T$, $k = T/M$ (so $k = \bigoplus_{\alpha \in \Gamma}T_{\alpha}/M_{\alpha}$ is a $\Gamma$-graded integral domain), $D =\bigoplus_{\alpha \in \Gamma} D_{\alpha}$ be a graded subring of $k$, $\varphi:T\to k$ be the canonical homomorphism, and $R= \varphi^{-1}(D)$. Then $R$ is a $\gGD$ domain if and only if both $T$ and $D$ are $\gGD$ domains.
\end{theorem}
In Section 4, we give some examples of $\gGD$ domains issued from pullbacks.


\subsection{Graded integral domains}

Let $\Gamma$ be a (nonzero) torsionless commutative cancellative monoid (written
additively) and $\langle \Gamma \rangle = \{a - b \mid a,b \in \Gamma\}$ be the
quotient group of $\Gamma$; so $\langle \Gamma \rangle$ is a torsionfree abelian group.
It is well known that a cancellative monoid $\Gamma$ is torsionless
if and only if it can be given a total order compatible with the monoid
operation \cite[page 123]{no68}. By a $(\Gamma$-)graded integral domain  $R =\bigoplus_{\alpha \in \Gamma}R_{\alpha}$,
we mean an integral domain graded by $\Gamma$.
That is, each nonzero $x \in R_{\alpha}$ has degree $\alpha$, i.e., deg$(x) = \alpha$,  and
deg$(0) = 0$. Thus, each nonzero $f \in R$ can be written uniquely as $f = x_{\alpha_1} + \dots + x_{\alpha_n}$ with
deg$(x_{\alpha_i}) = \alpha_i$ and $\alpha_1 < \cdots < \alpha_n$.
A nonzero $x \in R_{\alpha}$ for every $\alpha \in \Gamma$ is said
to be {\em homogeneous}.

Let  $H = \bigcup_{\alpha \in \Gamma}(R_{\alpha} \setminus \{0\})$; so
$H$ is the saturated multiplicative set of nonzero homogeneous elements of $R$.
Then $R_H$, called the {\em homogeneous quotient field} of $R$,
 is a $\langle \Gamma \rangle$-graded integral domain whose nonzero homogeneous elements are units.  The set of nonzero homogeneous elements of a graded domain $A=\bigoplus_{\alpha\in\Gamma}A_{\alpha}$ is denoted by $H(A)$. Hence $H=H(R)$.
We say that an overring $T$ of $R$ is a {\em homogeneous overring} of $R$ if
$T = \bigoplus_{\alpha \in \langle \Gamma \rangle}(T \cap (R_H)_{\alpha})$;
so $T$ is a $\langle \Gamma \rangle$-graded integral domain
such that $R \subseteq T \subseteq R_H$. We say that $R= \bigoplus_{\alpha\in\Gamma}R_{\alpha}$ is a {\em graded-valuation domain} (gr-valuation domain) if either $aR \subseteq bR$ or $bR \subseteq aR$ for all $a,b \in H$. The ring of fractions $R_S$ is a homogeneous overring of $R$ for
a multiplicative set $S$ of nonzero homogeneous elements of $R$
(with deg$(\frac{a}{b}) =$ deg$(a) -$ deg$(b)$ for $a \in H$ and $b \in S$).

The $\Gamma$-graded domain $R=\bigoplus_{\alpha\in\Gamma}R_{\alpha}$ is called a \emph{graded subring} of a $\Lambda$-graded domain $T =\bigoplus_{\alpha \in \Lambda}T_{\alpha}$, if $\Gamma$ is a subsemigroup of $\Lambda$ and for every $\alpha\in \Gamma$, $R_{\alpha}\subseteq T_{\alpha}$. It is easy to see that we can consider $R$ as a $\Lambda$-graded ring by setting $R_{\alpha}=0$ for each $\alpha\in\Lambda\setminus\Gamma$, and in this case $R=\bigoplus_{\alpha \in \Lambda}(T_{\alpha}\cap R)$. By ``$R\subseteq T$ is an extension of graded domains'' we mean that $R$ is a graded subring of $T$.

For a fractional ideal $A$ of $R$ with $A \subseteq R_H$,
let $A^*$ be the fractional ideal of $R$ generated by homogeneous elements in $A$; so
$A^* \subseteq A$. The fractional ideal $A$ is said to be {\em homogeneous} if $A^* = A$. A homogeneous ideal $P$ of $R$ is prime if and only if for all $a, b\in H$, $ab\in P$ implies $a\in P$ or $b\in P$ \cite[Page 124]{no68}. The set of all homogeneous prime ideals of $R$ is denoted by $\Spec_h(R)$. We note that a minimal prime ideal $P$ of a homogeneous ideal $I$ of $R$ is homogeneous. Indeed, the inclusions $I\subseteq P^*\subseteq P$ shows that $P^*=P$.

A homogeneous ideal of $R$ is called a {\em maximal homogeneous ideal}
if it is maximal among proper homogeneous ideals of $R$.
It is easy to see that each proper homogeneous ideal of $R$
is contained in a maximal homogeneous ideal of $R$. The set of all maximal homogeneous ideals of $R$ is denoted by $\Max_h(R)$.


Throughout this paper, let $\Gamma$ be a nonzero torsionless commutative cancellative monoid, $R=\bigoplus_{\alpha\in\Gamma}R_{\alpha}$ be an integral domain graded by $\Gamma$ and $H$ be the set of nonzero homogeneous elements of $R$.

\section{Preliminaries of pullback of graded domains}

Let $T=\bigoplus_{\alpha \in \Gamma}T_{\alpha}$ be a $\Gamma$-graded integral domain, $M =\bigoplus_{\alpha \in \Gamma}M_{\alpha}$ be a maximal homogeneous ideal of $T$, $k = T/M$ (so $k = \bigoplus_{\alpha \in \Gamma}T_{\alpha}/M_{\alpha}$ is a $\Gamma$-graded integral domain
in which each nonzero homogeneous element is invertible), $D =\bigoplus_{\alpha \in \Gamma} D_{\alpha}$ be a graded subring of $k$, $F$ be the homogeneous quotient field of $D$, $\varphi:T\to k$ be the canonical homomorphism, and $R= \varphi^{-1}(D)$ be the pullback of the following diagram:
\begin{displaymath}
\xymatrix{ R = \varphi^{-1}(D) \ar[r] \ar[d] &
D \ar[d] \\
T \ar[r]^{\varphi} & k = T/M. }
\end{displaymath}
It is shown that if $R_{\alpha} = \{x \in T_{\alpha} \mid \varphi(x) \in D_{\alpha}\}$ for each $\alpha \in \Gamma$, then $R =  \bigoplus_{\alpha \in \Gamma} R_{\alpha}$ is a graded domain \cite[Lemma 1]{cs18}. We refer to this diagram as a pullback of type $\bigtriangleup$. The following proposition collects the main properties of the pullbacks of type $\bigtriangleup$.

\begin{proposition}(see \cite[Proposition 2.1 and Remark 2.2]{cs18a})\label{pull}
In a pullback of type $\triangle$ we have:
\begin{enumerate}
\item $M$ is the conductor of $T$ to $R$; hence $M$ is divisorial in $R$.
\item $R$ and $T$ have the same (homogeneous) quotient fields; so $R_H = T_H$.
\item For each homogeneous prime ideal $P$ of $R$ such that
$M \nsubseteq P$, there exists a unique homogeneous prime ideal $Q$ of
$T$ such that $Q\cap R=P$. Moreover $R_{H \setminus P} = T_{H \setminus Q}$ and $R_P=T_Q$.
\item Let $Q$ be a maximal homogeneous ideal of $T$ with $Q \neq M$.
Then $R_{H \setminus Q \cap R} = T_{H \setminus Q}$ and $R_{Q \cap R} = T_Q$.
\item $M$ is a unique maximal homogeneous ideal of $T$ if and only if each homogeneous ideal of $R$ is
comparable to $M$.
\item The homogeneous ideals of $R$ containing $M$ are precisely the ideals $\varphi^{-1}(J)$, where $J$ is
a homogeneous ideal of $D$. Moreover, $R/\varphi^{-1}(J) = D/J$.
\item $M$ is the unique prime of $T$ contracting to $M$ in $R$, and $(R : M) = (M : M)$.
\item Suppose that $P$ is a homogeneous prime ideal of $R$ with
$P\supseteq M$. Then $T_{H \setminus P}=T_{H(T)\setminus M}$. In
particular, $T_{H \setminus M}=T_{H(T)\setminus M}$.
\item Assume that $k=F$ and set $N = H \setminus M$. Then
$D_{\varphi(N)}=k$, whence $R_{H \setminus M}=T_{H \setminus M}$ and
by (8) it is equal to $T_{H(T)\setminus M}$.
\end{enumerate}
\end{proposition}

The h-height of a homogeneous prime ideal $P$ of a graded ring $A$ (denoted by h-ht$_A(P)$) is defined to be the supremum of the lengths of chains of homogeneous prime ideals descending from $P$ and the h-dimension of $A$ (denoted by h-$\dim(A)$) is defined to be $\sup\{$h-$\mathrm{ht}_A(P)\mid P$ is a homogeneous prime ideal of $A\}$. In the following results we want to find the h-$\dim(R)$ in a pullback of type $\triangle$.

\begin{remark}\label{lll}
Consider a pullback diagram of type $\bigtriangleup$ and assume that $T$ has a unique maximal homogeneous ideal. Applying a similar argument as in \cite[Proposition 2.1(5)]{f80}, one can obtain the equality
$$\mathrm{h}\text{-}\dim(R)=\mathrm{h}\text{-}\dim(T)+\mathrm{h}\text{-}\dim(D).$$
\end{remark}

For the nongraded case of the following lemma see assertion (c) in proof of \cite[Theorem 1.4]{f80}.

\begin{lemma}\label{h} Consider a pullback diagram of type $\bigtriangleup$. Then $\text{h-}\mathrm{ht}_R(M)=\text{h-}\mathrm{ht}_T(M)$.
\end{lemma}
\begin{proof} First we prove the lemma in case $D=F$. Assume that $D=F$ and consider the following pullback of type $\triangle$ (see \cite[Remark 2.2(1)]{cs18a})
\begin{displaymath}
\xymatrix{ R_{H\setminus M} \ar[r] \ar[d] &
F \ar[d] \\
T_{H\setminus M} \ar[r] & k. }
\end{displaymath}
Hence by Remark \ref{lll} and Proposition \ref{pull}(9), we have $\text{h-}\mathrm{ht}_R(M)=\mathrm{h}\text{-}\dim(R_{H\setminus M})=\mathrm{h}\text{-}\dim(T_{H\setminus M})=\mathrm{h}\text{-}\dim(T_{H(T)\setminus M})=\text{h-}\mathrm{ht}_T(M)$.

In the general case, split the pullback diagram into two parts as follows
\begin{displaymath}
\xymatrix{ R \ar[r] \ar[d] &
D \ar[d] \\
S:=\varphi^{-1}(F) \ar[r] \ar[d] &
F \ar[d] \\
T \ar[r]^{\varphi} & k. }
\end{displaymath}
Then from the upper pullback and Proposition \ref{pull}(9), we have $\text{h-}\mathrm{ht}_R(M)=\text{h-}\mathrm{ht}_S(M)$, and from the lower one and the first paragraph we have $\text{h-}\mathrm{ht}_S(M)=\text{h-}\mathrm{ht}_T(M)$, to complete the proof.
\end{proof}

\begin{lemma}\label{ll}
Consider a pullback diagram of type $\bigtriangleup$ and assume that $k=F$. Then
$$\mathrm{h}\text{-}\dim(R)=\max\{\mathrm{h}\text{-}\mathrm{ht}_T(M)+\mathrm{h}\text{-}\dim(D),\mathrm{h}\text{-}\dim(T)\}.$$
\end{lemma}

\begin{proof} We have
$\text{h-}\dim(R)=\sup\{\text{h-}\dim(R_{H\setminus P})|P\in \Max_h(R)\}$. Let $P\in \Max_h(R)$ be such that $\text{h-}\dim(R)=\text{h-}\dim(R_{H\setminus P})$. We have two cases to consider:

\textbf{Case 1}: Assume that $P\not\supset M$. Then $R_{H\setminus P}=T_{H\setminus Q}$ for some $Q\in\Spec_h(T)$ such that $P=Q\cap R$ by Proposition \ref{pull}(3). Hence $\text{h-}\dim(R)=\text{h-}\dim(R_{H\setminus P})=\text{h-}\dim(T_{H\setminus Q})$. Now we claim that
h-$\dim(T)=\text{h-}\dim(T_{H\setminus Q})$. Suppose, contrary to our claim, that there exists $L\in \Max_h(T)$ such that h-$\dim(T)=\text{h-}\dim(T_{H\setminus L})$ and $\text{h-}\dim(T_{H\setminus Q})\lneq\text{h-}\dim(T_{H\setminus L})$. Set $P_1:=L\cap R$, which is a homogeneous prime ideal of $R$. Consequently
$R_{P_1}=T_L$ by \cite[Remark 2.2(5)]{cs18a} and hence using Proposition \ref{pull}(2) $R_{H\setminus P_1}=R_H\cap R_{P_1}=T_H\cap T_L=T_{H\setminus L}$. Thus
$\text{h-}\dim(R)=$h-$\dim(R_{H\setminus P})=$h-$\dim(T_{H\setminus Q})\lneq$h-$\dim(T_{H\setminus L})=$h-$\dim(R_{H\setminus P_1})$, which is a contradiction. It follows that h-$\dim(R)=\text{h-}\dim(R_{H\setminus P})=\text{h-}\dim(T_{H\setminus Q})=\text{h-}\dim(T)$. On the other hand by Proposition \ref{pull}(6) we have $D=R/M$; so that $\mathrm{h}\text{-}\mathrm{ht}_R(M)+\mathrm{h}\text{-}\dim(D)\leq\mathrm{h}\text{-}\dim(R)=\mathrm{h}\text{-}\dim(T)$. Note that by Lemma \ref{h}, we have $\text{h-}\mathrm{ht}_R(M)=\text{h-}\mathrm{ht}_T(M)$. Therefore in this case we have the equality
$$\mathrm{h}\text{-}\dim(R)=\mathrm{h}\text{-}\dim(T)=\max\{\mathrm{h}\text{-}\mathrm{ht}_T(M)+\mathrm{h}\text{-}\dim(D),\mathrm{h}\text{-}\dim(T)\}.$$

\textbf{Case 2}:  Assume that $P\supseteq M$ and consider the following pullback of type $\triangle$ (see \cite[Remark 2.2(1)]{cs18a})
\begin{displaymath}
\xymatrix{ R_{H\setminus P} \ar[r] \ar[d] &
D_{\varphi(H\setminus P)} \ar[d] \\
T_{H\setminus P} \ar[r] & k. }
\end{displaymath}
Hence h-$\dim(R_{H\setminus P})=$h-$\dim(D_{\varphi(H\setminus P)})+$h-$\dim(T_{H\setminus P})$ by Remark \ref{lll}. We claim that h-$\dim(D)=$h-$\dim(D_{\varphi(H\setminus P)})$. Suppose, contrary to our claim, that there exists $L\in\Max_h(D)$ such that h-$\dim(D)=$h-$\dim(D_{H(D)\setminus L})$ and h-$\dim(D_{\varphi(H\setminus P)})\lneq$h-$\dim(D_{H(D)\setminus L})$. Set $P_1:=\varphi^{-1}(L)$, which is a homogeneous prime ideal of $R$. Consequently
$$\text{h-}\dim(R_{H\setminus P_1})=\text{h-}\dim(D_{\varphi(H\setminus P_1)})+\text{h-}\dim(T_{H\setminus P_1})$$ by (choosing $N:=H\setminus P_1$ in) \cite[Remark 2.2(1)]{cs18a}, Remark \ref{lll}, and noting that $\varphi(H\setminus P_1)=\varphi(H)\setminus\varphi(P_1)=H(D)\setminus L$. Thus keeping in mind that $T_{H\setminus P}=T_{H\setminus M}=T_{H\setminus P_1}$ by Proposition \ref{pull}(8) we have
\begin{align*}
\text{h-}\dim(R)&=\text{h-}\dim(R_{H\setminus P})\\
&=\text{h-}\dim(D_{\varphi(H\setminus P)})+\text{h-}\dim(T_{H\setminus P})\\
& \lneq\text{h-}\dim(D_{H(D)\setminus L})+\text{h-}\dim(T_{H\setminus P})\\
&=\text{h-}\dim(D_{H(D)\setminus L})+\text{h-}\dim(T_{H\setminus P_1})\\
&=\text{h-}\dim(R_{H\setminus P_1}),
\end{align*}
which is a contradiction. Thus
\begin{align*}
\text{h-}\dim(R)&=\text{h-}\dim(R_{H\setminus P})\\
&=\text{h-}\dim(D)+\text{h-}\dim(T_{H\setminus M})\\
&=\text{h-}\dim(D)+\text{h-}\mathrm{ht}_T(M).
\end{align*}
Let $Q\in\Max_h(T)$ such that $\text{h-}\dim(T)=\text{h-}\dim(T_{H(T)\setminus Q})$. If $Q=M$, we have $\text{h-}\dim(T)=\text{h-}\mathrm{ht}_T(M)=\text{h-}\mathrm{ht}_R(M)\leq\text{h-}\dim(R)$, and if $Q\neq M$ we have $R_{Q\cap R}=T_Q$ by Proposition \ref{pull}(4). Thus $R_{H\setminus Q\cap R}=T_{H(T)\setminus Q}$ using Proposition \ref{pull}(2). Hence $\text{h-}\dim(T)=\text{h-}\dim(T_{H(T)\setminus Q})=\text{h-}\dim(R_{H\setminus Q\cap R})\leq\text{h-}\dim(R)$. Therefore here also we have the equality
$$\mathrm{h}\text{-}\dim(R)=\max\{\mathrm{h}\text{-}\mathrm{ht}_T(M)+\mathrm{h}\text{-}\dim(D),\mathrm{h}\text{-}\dim(T)\}.$$
\end{proof}

\begin{lemma}\label{lemf} For a pullback diagram of type $\triangle$ assume that $D=F$. Then  $\text{h-}\dim(R)=\text{h-}\dim(T)$.
\end{lemma}

\begin{proof} We have
$\text{h-}\dim(R)=\sup\{\text{h-}\dim(R_{H\setminus P})|P\in \Max_h(R)\}$. Let $P\in \Max_h(R)$ such that
$\text{h-}\dim(R)=\text{h-}\dim(R_{H\setminus P})$.

\textbf{Case 1}: Assume that $P=M$. Consider the following pullback of type $\triangle$ (see \cite[Remark 2.2(1)]{cs18a})
\begin{displaymath}
\xymatrix{ R_{H\setminus M} \ar[r] \ar[d] &
F \ar[d] \\
T_{H\setminus M} \ar[r] & k. }
\end{displaymath}
Then we have $\text{h-}\dim(R_{H\setminus M})=\text{h-}\dim(T_{H\setminus M})$ using Remark \ref{rem} below. We claim that $\text{h-}\dim(T_{H\setminus M})=\text{h-}\dim(T)$. Assume to the contrary that there exists a maximal homogeneous ideal $L$ of $T$ such that $\text{h-}\dim(T_{H\setminus M})\lneq\text{h-}\dim(T_{H\setminus L})=\text{h-}\dim(T)$. Hence by Proposition \ref{pull}(4) we have $R_{H\setminus L\cap R}=T_{H\setminus L}$. So that $$\text{h-}\dim(R)=\text{h-}\dim(T_{H\setminus M})\lneq\text{h-}\dim(T_{H\setminus L})=\text{h-}\dim(R_{H\setminus L\cap R}),$$ which is a contradiction. Therefore in this case we have $\text{h-}\dim(R)=\text{h-}\dim(T)$.

\textbf{Case 2}: Assume that $P\neq M$. Then $R_{H\setminus P}=T_{H\setminus Q}$ for some $Q\in\Spec_h(T)$ such that $P=Q\cap R$ by Proposition \ref{pull}(3). We claim that h-$\dim(T)=\text{h-}\dim(T_{H\setminus Q})$. Suppose, contrary to our claim, that there exists $L\in\Max_h(R)$ such that h-$\dim(T)=\text{h-}\dim(T_{H\setminus L})$ and
$\text{h-}\dim(T_{H\setminus Q})\lneq\text{h-}\dim(T_{H\setminus L})$. Set $P_1:=L\cap R$. If $P_1\not\supset M$ then $R_{H\setminus P_1}=T_{H\setminus L}$ by Proposition \ref{pull}(3). Thus $$\text{h-}\dim(R)=\text{h-}\dim(R_{H\setminus P})=\text{h-}\dim(T_{H\setminus Q})\lneq\text{h-}\dim(T_{H\setminus L})=\text{h-}\dim(R_{H\setminus P_1}),$$ which is again a contradiction. Therefore $$\text{h-}\dim(R)=\text{h-}\dim(R_{H\setminus P})=\text{h-}\dim(T_{H\setminus Q})=\text{h-}\dim(T).$$
\end{proof}

Splitting the pullback diagram of type $\triangle$ into two parts (as in the proof of Lemma \ref{h}), then using Lemmas \ref{h}, \ref{ll}, and \ref{lemf}, we have the following graded version of  \cite[Corollary 9]{br76} and \cite[Corollary 1.10]{gh00}.

\begin{proposition}\label{h-dim}
Consider a pullback diagram of type $\bigtriangleup$. Then
$$\mathrm{h}\text{-}\dim(R)=\max\{\mathrm{h}\text{-}\mathrm{ht}_T(M)+\mathrm{h}\text{-}\dim(D),\mathrm{h}\text{-}\dim(T)\}.$$
\end{proposition}

\section{When is a pullback a graded going-down domain?}

The pullback constructions are one of important tools to construct examples and counterexamples. In \cite[Theorem 2.4]{d09}, D. E. Dobbs introduced a portable property about integral domains and provided the behaviour of his property under pullbacks. The $\GD$ property is a special case of portable property. More precisely, he first did it for the CPI-extension (complete pre-image extension of Boisen and Sheldon \cite{bs77}). This section is devoted to investigate the $\gGD$ property for graded integral domains under graded pullbacks.

As Dobbs, we should consider the graded analogue of CPI-extension  being $\gGD$ domain. We recall the graded analogue of CPI-extension as follows. For a graded domain $R=\bigoplus_{\alpha\in\Gamma}R_{\alpha}$ and a homogeneous prime ideal $P$ of  $R$, it is the following pullback of type $\triangle$ (see \cite[Page 629]{cs18a}):

\begin{displaymath}
\xymatrix{ R(P):=\varphi^{-1}(R/P) \ar[r] \ar[d] &
R/P \ar[d] \\
R_{H\setminus P} \ar[r]^{\varphi} &
k=\frac{R_{H\setminus P}}{PR_{H\setminus P}}. }
\end{displaymath}

It can be seen that $R(P)=R+PR_{H\setminus P}$ is a graded
integral domain with $R(P)_{\alpha} = \{x \in R_{H \setminus P} \mid
\varphi(x) \in (R/P)_{\alpha}\}$ for all $\alpha \in \Gamma$.

To investigate the $\gGD$ property for the graded CPI-extension, we need the graded versions of \cite[Lemma 2.1 and Theorem 2.2]{d97}. Recall from \cite{ss23} that a homogeneous prime ideal $P$ of $R=\bigoplus_{\alpha\in\Gamma}R_{\alpha}$ is called a \emph{homogeneous divided prime ideal} if each element of $H\setminus P$ divides each element of $P$. It can be easily seen that $P$ is a homogeneous divided prime ideal if and only if $P$ is comparable to each homogeneous ideal of $R$ if and only if $P=PR_{H\setminus P}$.

\begin{lemma}\label{ggdtop}
Let $R=\bigoplus_{\alpha\in\Gamma}R_{\alpha}$ be a graded domain and $P$ be a homogeneous divided prime ideal of $R$. Then each extension of graded domains $R\subseteq T$ satisfies graded going-down to $P$.
\end{lemma}
\begin{proof} Suppose to the contrary that there exist homogeneous prime ideals $P\subseteq P_1$ of $R$ and $Q_1$ of $T$ satisfying $Q_1\cap R=P_1$, and no $Q\in\Spec_h(T)$ satisfying both $Q\subseteq Q_1$ and $Q\cap R=P$. Choose by \cite[Lemma 4.4]{co2018}, a gr-valuation overring $(V,N_1)$ of $T$ such that $N_1\cap T=Q_1$. Then $R\subseteq V$ does not satisfy graded going-down to $P$, since, if there exists $N\in\Spec_h(V)$ satisfying $N\cap R=P$, then $Q:=N\cap T$ would provide a contradiction. Accordingly, by \cite[Exercise 37, Page 44]{k74}, there exists $L\in\Spec_h(V)$ such that $L$ is a minimal prime ideal of $PV$ and
$$\Sigma p_iv_i=rz$$ for some $p_i\in P$, $v_i\in V$, $r\in R\setminus P$ and $z\in V\setminus L$.

Let $r=r_{{\alpha}_1}+\cdots+r_{{\alpha}_n}$ be the decomposition of $r$ into homogeneous components of degrees ${\alpha}_1,\ldots,{\alpha}_n$ with ${\alpha}_1<\cdots<{\alpha}_n$. Since $\Sigma p_iv_i-\Sigma_{r_{{\alpha}_j}\in P}r_{{\alpha}_j}z=\Sigma_{r_{{\alpha}_j}\notin P}r_{{\alpha}_j}z$, we can assume that $r_{{\alpha}_j}\in H\setminus P$ for $j=1,\ldots,n$. By decomposing $p_i$ and $v_i$ into homogeneous elements, we can assume that $p_i$ and $v_i$ are homogeneous elements. Now let $z=z_{{\beta}_1}+\cdots+z_{{\beta}_s}+\cdots+z_{{\beta}_m}$ be the decomposition of $z$ into homogeneous components of degrees ${\beta}_1,\ldots,{\beta}_m$ with ${\beta}_1<\cdots<{\beta}_m$, such that ${\beta}_s$ is the smallest degree such that $z_{{\beta}_s}\in V\setminus L$.
The right hand side of the equality
\begin{equation}\label{(0)}
    \Sigma p_iv_i=(r_{{\alpha}_1}+\cdots+r_{{\alpha}_n})(z_{{\beta}_1}+\cdots+z_{{\beta}_s}+\cdots+z_{{\beta}_m})
\end{equation}
is the sum of the entries of the following array:
$$
\begin{array}{ccccccc}
  r_{{\alpha}_1}z_{{\beta}_1} & r_{{\alpha}_1}z_{{\beta}_2} & r_{{\alpha}_1}z_{{\beta}_3} & \cdots & r_{{\alpha}_1}z_{{\beta}_s} & \cdots & r_{{\alpha}_1}z_{{\beta}_m} \\
  r_{{\alpha}_2}z_{{\beta}_1} & r_{{\alpha}_2}z_{{\beta}_2} & r_{{\alpha}_2}z_{{\beta}_3} & \cdots & r_{{\alpha}_2}z_{{\beta}_s} & \cdots & r_{{\alpha}_2}z_{{\beta}_m} \\
  r_{{\alpha}_3}z_{{\beta}_1} & r_{{\alpha}_3}z_{{\beta}_2} & r_{{\alpha}_3}z_{{\beta}_3} & \cdots & r_{{\alpha}_3}z_{{\beta}_s} & \cdots & r_{{\alpha}_3}z_{{\beta}_m} \\
  \vdots & \vdots & \vdots &  & \vdots &  & \vdots \\
  r_{{\alpha}_n}z_{{\beta}_1} & r_{{\alpha}_n}z_{{\beta}_2} & r_{{\alpha}_n}z_{{\beta}_3} & \cdots & r_{{\alpha}_n}z_{{\beta}_s} & \cdots & r_{{\alpha}_n}z_{{\beta}_m}.
\end{array}
$$
The component $r_{{\alpha}_1}z_{{\beta}_1}$ has the smallest degree $\alpha_1+\beta_1$ on both sides of the equality (\ref{(0)}).
By canceling   $\Sigma_{\deg(p_iv_i)=\alpha_1+\beta_1} p_iv_i=r_{{\alpha}_1}z_{{\beta}_1}$ from both sides of the equality (\ref{(0)}), we obtain the equality
\begin{equation}\label{(1)}
    \Sigma p_iv_i=\Sigma_{j+k\geq3}r_{{\alpha}_j}z_{{\beta}_k}.
\end{equation}
Now the component of the smallest degree in the right side of the equation (\ref{(1)}) is $r_{{\alpha}_1}z_{{\beta}_2}$ or $r_{{\alpha}_2}z_{{\beta}_1}$ or $r_{{\alpha}_1}z_{{\beta}_2}+z_{{\alpha}_2}t_{{\beta}_1}$. In any case we can cancel these components from both sides of the equality (\ref{(1)}) to obtain the equality
\begin{equation*}\label{(2)}
   \Sigma p_iv_i=\Sigma_{j+k\geq4}r_{{\alpha}_j}z_{{\beta}_k}.
\end{equation*}
Continuing this procedure we obtain an equality like
\begin{equation}\label{(3)}
    \Sigma p_iv_i=\Sigma_{j+k\geq s+1}r_{{\alpha}_j}z_{{\beta}_k}.
\end{equation}
Set $X:=\{r_{{\alpha}_j}z_{{\beta}_k}\mid \alpha_j+\beta_k=\alpha_1+\beta_s$ and $k\neq s\}$. Note that if $r_{{\alpha}_j}z_{{\beta}_k}\in X$, then $j\geq2$ and $k\leq s-1$. This means that $X\subseteq L$. The component of the smallest degree in the right side of the equation (\ref{(3)}) is one of $r_{{\alpha}_j}z_{{\beta}_k}$ with $\alpha_j+\beta_k=\alpha_1+\beta_s$ or a sum of some of these elements. If this component of the smallest degree is of the form $\Sigma_{x\in A}x$ for some $A\subseteq X$, then we can cancel this component from both sides of the equality (\ref{(3)}). After that if $\Sigma_{x\in B}x$ for some $B\subseteq X$ has the smallest degree, delete this component too and so on. Hence we can assume that
\begin{equation}\label{(4)}
    r_{{\alpha}_1}z_{{\beta}_s}+\Sigma_{x\in C}x=\Sigma p_iv_i
\end{equation}
has the smallest degree for some $C\subseteq X$. Since $PV\subseteq L$, we have $r_{{\alpha}_1}z_{{\beta}_s}\in L$. But $z_{{\beta}_s}\notin L$, and so that $r_{{\alpha}_1}\in L$. Since $L$ is a minimal prime ideal of $PV$, using \cite[Theorem 2.1]{Huck} there exist an element $y\in T\setminus Q$, and a nonnegative integer $n$ such that $r_{{\alpha}_1}^ny\in PV$. Note that we can assume that $y$ is a homogeneous element. Indeed let $y=y_{{\gamma}_1}+\cdots+y_{{\gamma}_p}+\cdots+y_{{\gamma}_{\ell}}$ be the decomposition of $y$ into homogeneous elements and assume that $y_{{\gamma}_p}\notin L$. Thus $r_{{\alpha}_1}^ny_{{\gamma}_1}+\cdots+r_{{\alpha}_1}^ny_{{\gamma}_p}+\cdots+r_{{\alpha}_1}^ny_{{\gamma}_{\ell}}=r_{{\alpha}_1}^ny\in PV$, and so that $r_{{\alpha}_1}^ny_{{\gamma}_p}\in PV$. Hence we can assume that $y$ is homogeneous. Thus $z_{{\beta}_s}y\in V\setminus L$ is a homogeneous element. Multiplying both sides of (\ref{(4)}) by $r_{{\alpha}_1}^{n}y$, we obtain
$$r_{{\alpha}_1}^{n+1}(z_{{\beta}_s}y)+r_{{\alpha}_1}^ny(\Sigma_{x\in C}x)=\Sigma r_{{\alpha}_1}^np_iv_iy.$$ This means that $r_{{\alpha}_1}^{n+1}(z_{{\beta}_s}y)\in PV$. Whence $r_{{\alpha}_1}^{n+1}(z_{{\beta}_s}y)=\Sigma p_iv_i$ for some homogeneous elements $p_i\in P$ and $v_i\in V$. Since $V$ is a gr-valuation domain, we may assume that $Vp_i\subseteq Vp_1$ for all $i$. Thus $p_i=a_ip_1$, with $a_i\in V$, it follows that $r_{{\alpha}_1}^{n+1}(z_{{\beta}_s}y)=p_1v$, with $v=\Sigma a_iv_i\in V$. Now, $p:=p_1(r_{{\alpha}_1}^{n+1})^{-1}\in PR_{H\setminus P}=P$, whence $z_{{\beta}_s}y=pv\in PV\cap(V\setminus L)=\emptyset$, a contradiction.
\end{proof}

For the proof of the following proposition we use \cite[Theorem 3.12]{ss23} that a graded domain $R$ is a $\gGD$ domain if and only if $R\subseteq V$ satisfies $\gGD$ for each gr-valuation overring $V$ of $R$.

\begin{proposition}\label{lem2}
Let $R=\bigoplus_{\alpha\in\Gamma}R_{\alpha}$ be a graded domain. The following are equivalent:
\begin{enumerate}
  \item There exists a divided homogeneous prime ideal $P$ of $R$ such that  $R/P$ and $R_{H\setminus P}$ are $\gGD$ domains.
  \item $R$ is a $\gGD$ domain.
\end{enumerate}
\end{proposition}
\begin{proof} (2)$\Rightarrow$(1) It is trivial by taking $P=0$.

(1) $\Rightarrow$(2) By Theorem \cite[Theorem 3.12]{ss23}, it is enough to show that $R\subseteq V$ satisfies $\gGD$ for each gr-valuation domain $V$. Assume that $P_0\subseteq P_1$ are homogeneous prime ideals of $R$ and $Q_1$ is a homogeneous prime ideal of $V$ such that $Q_1\cap R=P_1$. Note that $P$ is comparable to $P_0$ and $P_1$, since $P$ is a divided homogeneous prime ideal. We next consider three cases:

\textbf{Case 1:} $P\subseteq P_0$. Since, by Lemma \ref{ggdtop}, $R\subseteq V$ satisfies graded going-down to $P$, there exists a homogeneous prime $Q$ of $V$ such that $Q\subseteq Q_1$ and $Q\cap R=P$. It is easy to see that $R/P$ is a graded subring of $V/Q$ and that the extension $R/P\subseteq V/Q$ satisfies $\gGD$ property, since $R/P$ is a $\gGD$ domain by assumption. Note that $Q_1/Q\cap R/P=P_1/P$, there exists a homogeneous prime ideal $Q_0$ of $V$ such that $Q\subseteq Q_0\subseteq Q_1$ and $Q_0/Q\cap R/P=P_0/P$. Thus $Q_0\cap R=P_0$.

\textbf{Case 2:} $P_0\subseteq P\subseteq P_1$. By Case 1, there exists a homogeneous prime ideal $Q$ of $V$ such that $Q\subseteq Q_1$ and $Q\cap R=P$. Note that in this case $R_{H\setminus P}$ is a graded subring $V_{H(V)\setminus Q}$ and the extension $R_{H\setminus P}\subseteq V_{H(V)\setminus Q}$ satisfies graded going-down, since $R_{H\setminus P}$ is a gGD domain by assumption. As $QV_{H(V)\setminus Q}\cap R_{H\setminus P}=PR_{H\setminus P}$, there exists a homogeneous prime ideal $Q_0$ of $V$ such that $Q_0\subseteq Q$ and $Q_0V_{H(V)\setminus Q}\cap R_{H\setminus P}=P_0R_{H\setminus P}$. Thus $Q_0\cap R=P_0$.

\textbf{Case 3:} $P_1\subseteq P$. It is the same as Case 2, using the facts that $R_{H\setminus P}$ is a graded subring $V_{H\setminus P}$ and the extension $R_{H\setminus P}\subseteq V_{H\setminus P}$ satisfies graded going-down, to produce suitable $Q_0$ satisfying $Q_0\subseteq Q_1$ and $Q_0\cap R=P_0$.
\end{proof}

As an application we determine when the graded CPI-extension $R(P)$ is a $\gGD$ domain (see \cite[Corollary 2.3]{d97}).

\begin{corollary}\label{rpr}
Assume that $R=\bigoplus_{\alpha\in\Gamma}R_{\alpha}$ is a graded domain and $P$ is a homogeneous prime ideal of $R$. The following are equivalent:
\begin{enumerate}
  \item $R/P$ and $R_{H\setminus P}$ are $\gGD$ domains.
  \item $R(P)=R+PR_{H\setminus P}$ is a $\gGD$ domain.
\end{enumerate}
\end{corollary}
\begin{proof} By \cite[Lemma 4.6]{ss23}, $Q:=PR_{H\setminus P}$ is a homogeneous divided prime ideal of $R(P)$ and, by Proposition \ref{pull}(6), $R(P)/Q=R/P$. Also, by an easy computation, we have $R(P)_{H(R(P))\setminus Q}=R_{H\setminus P}$. Therefore (1)$\Rightarrow$(2) holds by Proposition \ref{lem2} and (2)$\Rightarrow$(1) holds by \cite[Corollary 4.8 and Proposition 4.1]{ss23}.
\end{proof}

The next lemma isolates the case of our main result where $T$ has a unique maximal homogeneous ideal and $D=F$ (see \cite[Proposition B.2]{ad80}).

\begin{remark}\label{rem}
{\em Consider a pullback diagram of type $\bigtriangleup$. Assume that $T=\bigoplus_{\alpha\in\Gamma}T_{\alpha}$ is a graded domain, with a unique maximal homogeneous ideal and $D=F$. Then $$\Spec_h(R)=\Spec_h(T)$$ as sets. Indeed assume that $P$ is a homogeneous prime ideal of $R$, $t\in T$ and $a\in P$. Then $(ta)^2=at^2a\in MTP=MP\subseteq P$. Since $ta\in M\subseteq R$, and $P$ is a prime ideal of $R$, we obtain that $ta\in P$. To show that $P$ is a prime ideal of $T$, assume that $ab\in P$ for some homogeneous elements $a,b\in T$. If $a,b\in M$, then  one of $a$ or $b$ is in $P$ as $P$ is a prime ideal of $R$. For the remaining case, suppose without loss of generality $a\notin M$. Then $a^{-1}\in T$ and $b=a^{-1}(ab)\in TP\subseteq P$ as required. Conversely, assume that $P$ is a homogeneous prime ideal of $T$, $r\in R$ and $a\in P$. Then $ra\in TP\subseteq P$. Hence $P$ is an ideal of $R$, and as $R\subseteq T$, we have in fact $P\in\Spec_h(R)$.
}
\end{remark}

\begin{lemma}\label{lem1}
Consider a pullback diagram of type $\bigtriangleup$. Assume that $T=\bigoplus_{\alpha\in\Gamma}T_{\alpha}$ is a graded domain, with a unique maximal homogeneous ideal $M =\bigoplus_{\alpha \in \Gamma}M_{\alpha}$, and $D=F$. Then $R$ is a $\gGD$ domain if and only if $T$ is a $\gGD$ domain.
\end{lemma}

\begin{proof} First of all note that $R$ has a unique maximal homogeneous ideal $M$, since $R/M=F$ and Proposition \ref{pull}(6).

Assume that $R$ is a $\gGD$ domain. To show that $T$ is a $\gGD$ domain, let $V$ be a gr-valuation overring  of $T$, $P_0\subseteq P_1$ be homogeneous prime ideals of $T$, and $Q_1$ be a homogeneous prime ideal of $V$ such that $Q_1\cap T=P_1$. Since $P_0\subseteq P_1$ are homogeneous prime ideals of $R$ by Remark \ref{rem}, and so $Q_1\cap R=P_1\cap R=P_1$. Thus there exists a homogeneous prime ideal $Q_0$ of $V$ such that $Q_0\subseteq Q_1$ and $Q_0\cap R=P_0$. Since $Q_0\cap T\subseteq Q_1\cap T=P_1\subseteq M\subseteq R$, we obtain $P_0=Q_0\cap R=(Q_0\cap T)\cap R=Q_0\cap T$. Thus $T$ is a $\gGD$ domain by \cite[Theorem 3.12]{ss23}.

Conversely, assume that $T$ is a $\gGD$ domain and let $V$ be a homogeneous overring of $R$, $P_0\subseteq P_1$ are homogeneous prime ideals of $R$ and $Q_1$ is a homogeneous prime ideal of $V$ such that $Q_1\cap R=P_1$. By localization to $H(V)\setminus Q_1$ we can assume that $V$ has a unique homogeneous maximal ideal $Q_1$. Also note that $P_0$ and $P_1$ are prime ideals of $T$ by Remark \ref{rem}. Now let $S:=TV$, which is a $\Gamma$-graded domain in the way that $S_{\gamma}=\Sigma_{\gamma=\alpha+\beta}T_{\alpha}V_{\beta}$ for each $\gamma\in\Gamma$. Then $P_1S=(P_1T)V=P_1V\subseteq Q_1$, and so $P_1S\neq S$. Let $W_1$ be a minimal prime ideal of $S$ containing $P_1S$. Since $T\subseteq S$ satisfies $\gGD$, we obtain that $W_1\cap T=P_1$. So there exists a homogeneous prime ideal $W_0$ of $S$ such that $W_0\subseteq W_1$ and $W_0\cap T=P_0$. Set $Q_0:=W_0\cap V$. Then $Q_0$ is a homogeneous prime ideal of $V$ such that $Q_0\cap R=P_0$ and $Q_0\subseteq Q_1$ since $V$ has a unique homogeneous maximal ideal $Q_1$. Thus again $R$ is a $\gGD$ domain by \cite[Theorem 3.12]{ss23}.
\end{proof}

The next result drops the assumption that $D=F$ in the above lemma (see \cite[Lemma 2.3]{d09}).

\begin{lemma}\label{loc}
Consider a pullback diagram of type $\bigtriangleup$. Assume that $T$ has a unique homogeneous maximal ideal $M$. Then $R$ is a $\gGD$ domain if and only if both $T$ and $D$ are $\gGD$ domains.
\end{lemma}

\begin{proof} Let $F$ denote the homogenous quotient field of $D$. Put $S:=\varphi^{-1}(F)$. Suppose first that $S=T$, i.e., that is $F=T/M=k$. Then $R/M=D$ and $R_{H\setminus M}=T_{H\setminus M}=T$ by Proposition \ref{pull}(9). Hence, $R=R + MR_{H\setminus M}=R(M)$. Now by Corollary \ref{rpr}, $R$ is a $\gGD$ domain if and only if $R_{H\setminus M}=T$ and $R/M=D$ are $\gGD$ domains.

In the remaining case, $S \subsetneq T$. As we have seen, it follows from Lemma \ref{lem1} that $T$ is a $\gGD$ domain if and only if $S$ is a $\gGD$ domain. Moreover, $S$ has a unique homogeneous maximal ideal $M$ by Proposition \ref{pull}(6). Thus, by replacing $T$ with $S$, it follows from that analysis that, $R$ is a $\gGD$ domain if and only if both $S$ and $D$ are $\gGD$ domains. This holds if and only if both $T$ and $D$ are $\gGD$, as asserted.
\end{proof}

We next present our main result of this paper. For the proof we use \cite[Proposition 4.1]{ss23} that a graded domain $R$ is a $\gGD$ domain if and only if
$R_{H\setminus P}$ is a $\gGD$ domain for all $P\in \Spec_h(R)$ if and only if $R_{H\setminus M}$ is a $\gGD$ domain for all $M\in \Max_h(R)$. The nongraded version of the result is \cite[Theorem 2.4]{d09}.

\begin{theorem}\label{m}
Consider a pullback diagram of type $\bigtriangleup$. Then $R$ is a $\gGD$ domain if and only if both $T$ and $D$ are $\gGD$ domains.
\end{theorem}

\begin{proof} By \cite[Proposition 4.1]{ss23}, $\gGD$ is a local property. Assume that $P$ is a maximal homogeneous ideal of $R$. If $M\nsubseteq P$, then by Proposition \ref{pull}(3), there exists a unique homogeneous prime ideal $\fm$ of $T$ such that $P=\fm\cap R$ and $R_{H\setminus P}=T_{H\setminus \fm}=T_{H(T)\setminus \fm}$. Indeed, by Proposition \ref{pull}(2), we have $R_H=T_H=T_{H(T)}$ and $R_P=T_{\fm}$; so that $R_{H\setminus P}=R_H\cap R_P=R_H\cap T_{\fm}=T_{H(T)}\cap T_{\fm}=T_{H(T)\setminus \fm}$. Note that from the uniqueness of $\fm$ and maximality of $P$, we have $\fm\in\Max_h(T)$ and of course $\fm\neq M$.

On the other hand, if $M\subseteq P$, then $T_{H\setminus P}=T_{H(T)\setminus M}$ by Proposition \ref{pull}(8). Since $D=R/M$, we then have $N:=P/M\in\Max_h(D)$. Indeed, as $P$ varies over the maximal homogeneous ideals of $R$ that contain $M$, we
have that $N$ varies over the maximal homogeneous ideals of $D$. Note that $\varphi(H\setminus P)=\varphi(H)\setminus \varphi(P)=H(D)\setminus N$. Hence we have the following pullback diagram of type $\triangle$ by \cite[Remark 2.2(1)]{cs18a},
\begin{displaymath}
\xymatrix{ R_{H\setminus P} \ar[r] \ar[d] &
D_{H(D)\setminus N} \ar[d] \\
T_{H(T)\setminus M} \ar[r]^{\varphi} & k. }
\end{displaymath}
Therefore by Lemma \ref{loc}, $R_{H\setminus P}$ is a $\gGD$ domain, if and only if $T_{H(T)\setminus M}$ and $D_{H(D)\setminus N}$ are $\gGD$ domains.

Let $\fm$ be a homogeneous maximal ideal of $T$ such that $\fm\neq M$. Then by Proposition \ref{pull}(4), we have $R_{\fm\cap R}=T_{\fm}$. It follows that $R_{H\setminus\fm\cap R}=T_{H\setminus\fm}=T_{H(T)\setminus\fm}$.

Therefore from the above paragraphs we obtain that $R$ is a $\gGD$ domain if and only if $R_{H\setminus P}$ is a $\gGD$ domain for each $P\in \Max_h(R)$, if and only if $T_{H(T)\setminus \fm}$ and $D_{H(D)\setminus N}$ are $\gGD$ domains for each $\fm\in \Max_h(T)$ and $N\in \Max_h(D)$ if and only if both $T$ and $D$ are $\gGD$ domains.
\end{proof}

\section{Examples of graded going-down domains}

In this section we give examples of $\gGD$ domains using pullback of graded domains.

Recall that a graded domain $R$ is a \emph{graded-Pr\"{u}fer domain} if each finitely generated homogeneous ideal of $R$ is invertible. Consider a pullback of type $\triangle$. It is known from \cite[Corollary 3.8]{cs18a} that $R$ is graded-Pr\"{u}fer domain if and only if $T$ and $D$ are graded-Pr\"{u}fer domain, $M$ is a maximal $t$-ideal of $T$, and $k=F$.

Note that every graded-Pr\"{u}fer domain is a $\gGD$ domain \cite[Corollary 3.8]{ss23}. The following example generates a new class of $\gGD$ domains which are not graded-Pr\"{u}fer.

\begin{example}
{\em Let $A = \bigoplus_{\alpha \in \Gamma}A_{\alpha}$
 be a graded-Pr\"{u}fer domain, $P = \bigoplus_{\alpha \in \Gamma}P_{\alpha}$ be a maximal homogeneous ideal of $A$,
$T = A[X,X^{-1}]$ be a $(\Gamma \times \mathbb{Z})$-graded integral domain with deg$(aX^n) = (\alpha, n)$ for $0 \neq a \in A_{\alpha}$ and $n \in \mathbb{Z}$, $M = P[X,X^{-1}]$, and $D = A/P$. Then $D$ and $T$ are $\gGD$ domains, $k = (A/P)[X,X^{-1}](\neq F)$, and
$R = A + P[X, X^{-1}]$, i.e., $R = \bigoplus_{(\alpha, n) \in \Gamma \times \mathbb{Z}}R_{(\alpha, n)}$ with

\[
R_{(\alpha, n)} = \left \{
\begin{array}{ll}
A_{\alpha},    & \mbox{if $n=0$ }
\\P_{\alpha} X^n,   & \mbox{if $n\neq0$.}
\end{array}
\right.\]
Thus $R$ is a $\gGD$ domain by Theorem \ref{m}, while it is not a graded-Pr\"{u}fer domain by \cite[Corollary 3.8]{cs18a} }
\end{example}

In the following, we give a concrete example of a $\gGD$ domain which is not graded-Pr\"{u}fer.

\begin{example}
{\em Let $\mathbb{Q}$ be the field of rational numbers and $X$ and $Y$ are algebraically independent indeterminates over $\mathbb{Q}$. Consider the following pullback of $\mathbb{N}\cup\{0\}$-graded domains of type $\triangle$:
\begin{displaymath}
\xymatrix{ \mathbb{Q}+Y\mathbb{Q}(X)[Y] \ar[r] \ar[d] &
\mathbb{Q} \ar[d] \\
\mathbb{Q}(X)[Y] \ar[r] & \frac{\mathbb{Q}(X)[Y]}{(Y)}=\mathbb{Q}(X). }
\end{displaymath}
Here $R:=\mathbb{Q}+Y\mathbb{Q}(X)[Y]$ is a $\gGD$ domain by Theorem \ref{m}, but it is not a graded-Pr\"{u}fer domain by \cite[Corollary 3.8]{cs18a}.}
\end{example}

Recall that $R$ is called a \emph{gr-Noetherian domain} if the homogeneous ideals of $R$ satisfy the ascending chain condition.

Consider a pullback of type $\triangle$. It is known from \cite[Theorem 2]{cs18} that $R$ is gr-Noetherian if and only if $T$ is gr-Noetherian, $D=F$, and $k$ is a free $F$-module of finite rank.

It is shown that a gr-Noetherian domain $R$ is a $\gGD$ domain if and only if h-$\dim(R)\leq1$ \cite[Proposition 3.13]{ss23}.  The following example shows the gr-Noetherian assumption is crucial.

\begin{example}
{\em Let $\mathbb{Q}$ be the field of rational numbers and $X$ and $Y$ are algebraically independent indeterminates over $\mathbb{Q}$. Consider the following pullback of $\mathbb{N}\cup\{0\}$-graded domains of type $\triangle$:
\begin{displaymath}
\xymatrix{ \mathbb{Z}_{(2)}+Y\mathbb{Q}(X)[Y] \ar[r] \ar[d] &
\mathbb{Z}_{(2)} \ar[d] \\
\mathbb{Q}(X)[Y] \ar[r] & \frac{\mathbb{Q}(X)[Y]}{(Y)}=\mathbb{Q}(X). }
\end{displaymath}
Then $R:=\mathbb{Z}_{(2)}+Y\mathbb{Q}(X)[Y]$ is a $\gGD$ domain by Theorem \ref{m}, and it is not a gr-Noetherian domain by \cite[Theorem 2]{cs18}. Also h-$\dim(R)=2$ by Proposition \ref{h-dim}.}
\end{example}
\vspace{.2cm}

\noindent {\bf Acknowledgement.} The authors would like to thank the referee for his/her careful reading of the manuscript and several comments which greatly improved the paper. This paper is published as part of research supported by the Research Affairs Office of the University of Tabriz.

\end{document}